\documentclass[12pt]{article}
\usepackage{amsfonts,mathtools,amsthm,hyperref}
\usepackage[english]{babel}
\usepackage{float}
\usepackage[margin=2.6cm]{geometry}
\usepackage{graphicx}
\newtheorem{theorem}{Theorem}
\newtheorem{lemma}{Lemma}
\newtheorem{proposition}{Proposition}

\newtheorem{remark}{Remark}
\newtheorem{definition}{Definition}
\newtheorem{corollary}{Corollary}

\newcommand\eps{\varepsilon}
\renewcommand\P{\mathbb{P}}
\newcommand\R{\mathbb{R}}
\newcommand\Z{\mathbb{Z}}

\newcommand\F{\mathcal{F}}

\newcommand\g{\gamma}
\newcommand\be{\mathbf{e}}
\DeclareMathOperator{\card}{card}
\DeclarePairedDelimiter{\norm}{\lVert}{\rVert}

\title{Two-dimensional Rademacher walk}

\author{Satyaki Bhattacharya\thanks{Centre for Mathematical Sciences, Lund University, Box 118 SE-22100, Lund, Sweden.} \ and Stanislav Volkov${}^*$}

\begin{document}
\maketitle
\begin{abstract}
We study a generalisation of the one-dimensional Rademacher random walk introduced in~\cite{BV} to $\Z^2$ (for $d\ge 3$, the Rademacher random walk is always transient, as follows from Theorem~8.8 in~\cite{EV}). This walk is defined as the sum of a sequence of independent steps, where each step goes in one of the four possible directions with equal probability, and the size of the $n$th step is $a_n$ where $\{a_n\}$ is a given sequence of positive integers. We establish some general conditions under which the walk is recurrent or transient.
\end{abstract}



\section{Introduction}
A one-dimensional Rademacher walk was introduced in~\cite{BV}, as a generalisation of a simple symmetric random walk on $\Z$ with step sizes following some pre-determined sequence of integers $a_1,a_2,\dots$ (as opposed to all step sizes being one). It was shown that this walk is recurrent for some very slowly growing sequences, and transient, for example, if $a_k=\lfloor k^\gamma\rfloor$ for all $\gamma>0$. Some further generalisations of the Rademacher walk with non-integer step sizes were also considered, and they were studied further in~\cite{Bristol}. This paper studies the extension of the walk to higher dimensions.

We say that $\xi$ has a {\em two-dimensional Rademacher distribution}, if $\xi$ takes one of  the values $\{\be_1,\be_2,-\be_1,-\be_2\}$ where $\be_1=(1,0)$, $\be_2=(0,1)$ are the unit vectors in $\R^2$ with equal probability~$1/4$.

Let $\xi_n$, $n=1,2,\dots$ be a sequence of i.i.d.\ two-dimensional Rademacher random variables. Let~$a_n$ be a deterministic sequence of positive real numbers. The two-dimensional Rademacher walk is defined as $S_n=Z_1+Z_2+\dots+Z_n$ where $Z_i=a_i \xi_i$.

\begin{definition}\label{defTR}
We say that the walk $S_n$ is transient if $\norm{S_n}\to\infty$ a.s. We say that the walk $S_n$ is {\em recurrent}, if $S_n=z$ for infinitely many $n$ a.s., for each $z\in\Z^2$.
\end{definition}

Note that the transience and recurrence are zero-one events. Also, in general, we do not have the dichotomy between recurrence vs.\ transience; for example, if all $a_n\equiv 2$ then the walk is neither recurrent nor transient according to our definition, as it returns to the origin (denoted by $\mathbf{0}=(0,0)$\,) infinitely often, yet never visits points with odd coordinates.

We study the conditions for transience in Section~\ref{SecT} and for recurrence in Section~\ref{SecR}.

\section{Transience}\label{SecT}
\subsection{Non-decreasing steps}
\begin{theorem}\label{t1}
Suppose that $a_n$ is a non-decreasing sequence of integers converging to infinity. Then $S_n$ is transient.
\end{theorem}

First, we prove a necessary technical result.
\begin{lemma}\label{modlemma}
Let $d_1,d_2,d_3\dots$ be a sequence of distinct positive integers, and let $T_k=d_1\eps_1+\dots+d_k\eps_k$ where  $\eps_i$ are i.i.d.\ Rademacher. Then for every integer $\displaystyle  m\ge \max_{1\le i\le k}d_i$ and every  $M\in\Z$ we have
$$
\P(T_k-M\equiv 0\pmod m)\le \frac{C \log k}{k}    
$$
for  some universal constant $C>0$.
\end{lemma}
\begin{proof}
It is sufficient to prove the statement for $M\in\{0,1,\dots,m-1\}$, which we will assume from now on. By the result of~\cite{SAR}, for each $z\in \Z$ we have $\P(T_k=z)\le \frac{C_1}{k^{3/2}}$. On the other hand, by Hoeffding's inequality~\cite{Hoe}
$$
\P(|T_k|\ge m \sqrt{2k\log k})\le 2\exp\left(-\frac{4\,m^2\,k\, \log k}{4d_1^2+\dots+4d_k^2}\right)\le2\exp(-\log k)=\frac{2}{k}.
$$
Consequently,
\begin{align*}
\P(T_k\equiv M\pmod m)&\le \P(T_k\equiv M\pmod m, |T_k|<m\sqrt{2k\log k})+
\P(|T_k|\ge m \sqrt{2k\log k})
\\ &
\le 
\P(T_k=M+\ell m\text{ for some $\ell\in\Z$ s.t. } |M+m\ell|<m \sqrt{2k\log k})+\frac2{k}
\\ &
\le \frac{1+2m \sqrt{2k\log k}}{m}\sup_{z\in\Z} \P(T_k=z)+\frac2{k}
\le \frac{3m \sqrt{k\log k}}{m}\times\frac{C_1}{k^{3/2}}+\frac2{k}
\\ &
= \frac{3C_1 \sqrt{\log k}}{k}+\frac2{k}\le\frac{C\log k}{k}
\end{align*}
for $C=3C_1+2>0$.
\end{proof}

\begin{corollary}\label{modcor}
Let $d_1,d_2,d_3\dots,d_n$ be a sequence of positive integers, of which at least $k$ are distinct, say $d_{\ell_1},d_{\ell_2},\dots,d_{\ell_k}$, and let $T_n=\sum_{i=1}^n d_i\eps_i$ where  $\eps_i$ are i.i.d.\ Rademacher. Then for every integer $\displaystyle  m\ge \max_{1\le i\le k}d_{\ell_i}$ we have
$$
\P(T_n\equiv 0\pmod m)\le \frac{C \log k}{k}    
$$
where $C$ is the constant from the statement of Lemma~\ref{modlemma}.
\end{corollary}
\begin{proof}
Suppose w.l.o.g.\ that $d_1,d_2,\dots,d_k$ are distinct. Let $T_n=\tilde T_k+U$ where $\tilde T_k=\sum_{i=1}^k d_i\eps_i$. Then, since $\tilde T_k$ and $U$ are independent,
\begin{align*}
\P(T_n\equiv 0\pmod m)&=\sum_M  \P(\tilde T_k-M\equiv 0\pmod m)\,\P(U=M)
\\ &\le \left(\sup_M \P(\tilde T_k-M\equiv 0\pmod m) \right) \sum_M \P(U=M)
\le \frac{C \log k}{k} 
\end{align*}
by Lemma~\ref{modlemma}.
\end{proof}

\begin{proof}[Proof of Theorem~\ref{t1}]
Let $b_1=a_1\ge 1$, $b_2=\min\{a_n: a_n>b_1\}\ge a_1+1$,  $b_3=\min\{a_n: a_n>b_2\}\ge a_2+1$, etc. Denote by $m_i=\card\{n:\ a_n=b_i\}$, $i=1,2,\dots$. Then the sequence $\{a_n\}$ can be represented as
$$
\underbrace{b_1, b_1, \dots b_1,}_{m_1}
\underbrace{b_2, b_2, \dots b_2,}_{m_2} 
\underbrace{b_3, b_3, \dots b_3,}_{m_3} 
\dots
\underbrace{b_j, b_j, \dots b_j,}_{m_j}b_{j+1},\dots .
$$
Let  $k_{j+1}=m_1+m_2+\dots+m_{j}+1=\inf\{n:\ a_n=b_{j+1}\}$ be the first index of the sequence $\{a_n\}$ equal to $b_{j+1}$.

Let $\kappa_i=1$ if $\xi_i\in\{\be_1,-\be_1\}$ and $=0$ otherwise. Let $\eps_i=\pm 1$ equal to the non-zero coordinate of~$\xi_i$, Then $\{\kappa_i\}$ are i.i.d.\ Bernoulli and $\{\eps_i\}$ are i.i.d.\ Rademacher; moreover, these two sequences are independent. The values of $i\in\{1,2,\dots,n\}$ for which $\kappa_i=1$ ($=0$ resp.) correspond to horizontal (vertical, resp.) steps. Formally, let
$$
H_n=\{1\le i<n:\ \kappa_i=1\},\quad
V_n=\{1\le i< n:\ \kappa_i=0\}
$$
be the indices corresponding to horizontal (resp.\ vertical) steps. Note that $Z_i=X_i \be_1+Y_i\be_2$ where 
$$
X_i=\begin{cases}
   a_i\eps_i,&\text{if } \kappa_i=1;\\
   0, &\text{if } \kappa_i=0
\end{cases}
\qquad
Y_i=\begin{cases}
   0,&\text{if } \kappa_i=1;\\
   a_i\eps_i,&\text{if } \kappa_i=0,
\end{cases}
$$
and denote 
$$
S_n=S^X_n\be_1+S^Y_n\be_2\quad \text{ where }
S_n^X=X_1+\dots+X_n,\quad S_n^Y=Y_1+\dots+Y_n.
$$
Fix a large integer $j$ and let
$$
A_j=\left\{\card\{a_i,i\in H_{k_j}\}<\frac j3\text{ \ or \ }\card\{a_i,i\in V_{k_j}\}<\frac j3\right\}.
$$
For each $1\le i< k_j$, $\kappa_i$ equals $0$ or $1$ with equal probability $\frac12$. Hence, by Hoeffding's inequality, there exists a constant $c_0$ such that  
$$
\P(A_j)\le 2e^{-c_0 j},
$$ 
i.e., there are at least $j/3$ horizontal and $j/3$ vertical steps with probability tending to one. As a result, on $A_j^c$, we can write
$$
S^X_{k_j-1}=\sum_{i=1}^{h_j} b'_i \eps'_i, \quad S^Y_{k_j-1}=\sum_{i=1}^{v_j} b''_i \eps''_i
$$
where $h_j=\card\{H_{k_j}\}$, $v_j=\card\{V_{k_j}\}$, the positive integer sequences $\{b_i'\}_{i=1}^{h_j}$ and $\{b_i''\}_{i=1}^{v_j}$ are non-decreasing and each has at least $j/3$ {\em distinct} elements, while $\eps_i',\eps_i''$ are i.i.d.\ Rademacher.

Now, by Corollary~\ref{modcor},
\begin{align*}
\P(S^X_{k_j-1}\equiv 0\pmod {b_j} \mid A_j^c)\le \frac{C\log (j/3)}{j/3},\qquad
\P(S^Y_{k_j-1}\equiv 0\pmod {b_j} \mid A_j^c)\le \frac{C\log (j/3)}{j/3}
\end{align*}
and, moreover, conditioned on the sequence $\kappa_i$ and $A_j^c$, the events above are independent. Since for the walk to hit $\mathbf{0}$ at time $n$, $k_{j}\le n\le k_{j+1}-1$, we need that both its coordinates are divisible by~$b_j$ at time~$k_j-1$, we have
\begin{align}\label{eqBC}
\P(S_n=\mathbf{0}\text{ for some }n\in [k_j,k_{j+1}))
\le \P(A_j)+\left(\frac{C\log (j/3)}{j/3}\right)^2
=O\left(\frac{\log^2 j}{j^2}\right).
\end{align}
Since the quantity in~\eqref{eqBC} is summable in $j$, by the Borel-Cantelli lemma we conclude that the event $\{S_n=\mathbf{0}\}$ occurs finitely often. By the same token, we get that $\{S_n=\mathbf{w}\}$ occurs finitely often for each ${\bf w}\in \Z^2$. Hence, the transience follows.
\end{proof}

\subsection{Transience for non-integer sequences \texorpdfstring{$\{a_n\}$}{}}
In this section, we allow $a_n$ to take any positive real values (not only integers). We will use the same Definition~\ref{defTR} for the transience of the walk.

\begin{definition}
Let $a_n\in \R_+$,  $n=1,2,\dots$. We call the sequence $\{a_n\}$ {\em good}, if for each large~$n$ there is a positive integer $K_n$ such that
$$
\frac{K_n}{\log_2 a_n}\to 1\text{ as }n\to\infty,
$$
and a subset  of $\{1,2,\dots,n\}$  with $K_n$ distinct elements, say $\{ i_1,i_2,\dots,i_{K_n}\}$, such that $i_{K_n}= n$ and $a_{i_{k+1}}\ge 2a_{i_{k}}$ for every $k=1,2,\dots,K_n-1$.
\end{definition}

Later in the Section, in Proposition~\ref{propgood}, we present a sufficient condition for a sequence to be good.

\begin{definition}
Let $r,s\ge 1$ be some constants. We call the sequence $\{a_n\}$ {\em $(r,s)$-monotone}, if for all large~$n$  we have $a_n\le s a_m$ whenever $m\ge rn$.
\end{definition}
\begin{remark}
The case $r=s=1$ corresponds in the above definition to monotonicity in the usual sense.
\end{remark}

\begin{theorem}\label{Tdouble}
Assume that the sequence $a_n$ is good and $(r,s)$-monotone, and for some $\epsilon>0$ we have $a_n\ge (\log n)^{1/2+\epsilon}$ for all sufficiently large $n$. Then the two-dimensional Rademacher walk is transient.
\end{theorem}
\begin{proof}
To show transience, it suffices to show that for each $(x,y)\in \R^2$ we have $S_n\in (x,x+1)\times(y,y+1)$ for at most finitely many $n$s a.s. W.l.o.g.\ we can assume that $a_n> 1$ for all $n$ (this is, of course, true for all large $n$ due to the assumption of the theorem, and a few initial values of the sequence $\{a_n\}$ do not affect transience).

The proof will follow an argument involving anti-concentration inequalities on two different scales, similarly to the method used in~\cite{Bristol}. Define the event 
\begin{align*}
A_n=\{&\text{The walk makes at least $n/5$ horizontal and $n/5$ vertical steps} 
 \\
 & \text{among the last $n/2$ steps of the $n$--step walk}\}
\end{align*}
By Hoeffding's inequality, it is easy to check that $\P(A_n^c)\le 2e^{-0.01 n}$.

Let  $T_m$ be a one-dimensional Rademacher random walk with step sizes $d_1,d_2,\dots,d_m\ge D$ for some $D>0$.  By the Erd\H{o}s--Littlewood--Offord inequality (see~\cite{Erdos} or Lemma~14 in~\cite{Bristol}),
\begin{align}\label{eqELO}
\P(T_m \in (x-D,x+D])\leq\frac{0.8}{\sqrt{m}}  
\end{align}
for all $x\in\R$ and $m\in \Z_+$.

Let $n_1:=\lfloor n/2\rfloor$ and consider the walk $S_n$ during the times from $n_1+1$ to~$n$. Define 
\begin{align}\label{eqDn}
D=D_n
= \min_{n_1< i\le n}a_i .
\end{align}
Let
\begin{align*}  
A_n'&=\left\{\sum_{i=n_1+1}^n Z_i\cdot \be_1\in \left(x-D,x+D\right]\right\}
\\
A_n''&=\left\{\sum_{i=n_1+1}^n Z_i\cdot \be_2\in \left(x-D,x+D\right]\right\}
\end{align*}
Conditionally on $A_n$ and the sequence $(\kappa_i)$ defined in the proof of Theorem~\ref{t1}, these two events are independent, and by~\eqref{eqELO} each has a (conditional) probability less than $\frac{0.8}{\sqrt{n/5}}$. Consequently, for some $c_1>0$ we have
\begin{align}\label{eqS'}
\P\left(S'\in (x-D,x+D)\times(y-D,y+D)\right)
\le \frac{c_1^2}n+\P(A_n^c)=\frac{c_1^2+o(1)}n.    
\end{align}
where $S'=\sum_{i=n_1+1}^n Z_i$.

\vspace{5mm}
Now let us observe the walk during its first $n_0:=\left\lfloor \frac{n_1}{r}\right\rfloor\le n_1$ steps. By the assumption of the theorem, it is possible to extract a subsequence $i_1,i_2,\dots,i_K$ from $\{1,2,\dots,n_0\}$ of size
$$
K=(1+\g_{n_0}) \log_2 a_{n_0} \quad\text{where } \g_{n_0}\to 0
$$
such that $a_{i_{k+1}}\ge 2 a_{i_{k}}$ for $k=1,2,\dots,K-1$, and $i_K=n_0$.

\begin{figure}
    \centering
    \includegraphics[width=0.7\linewidth]{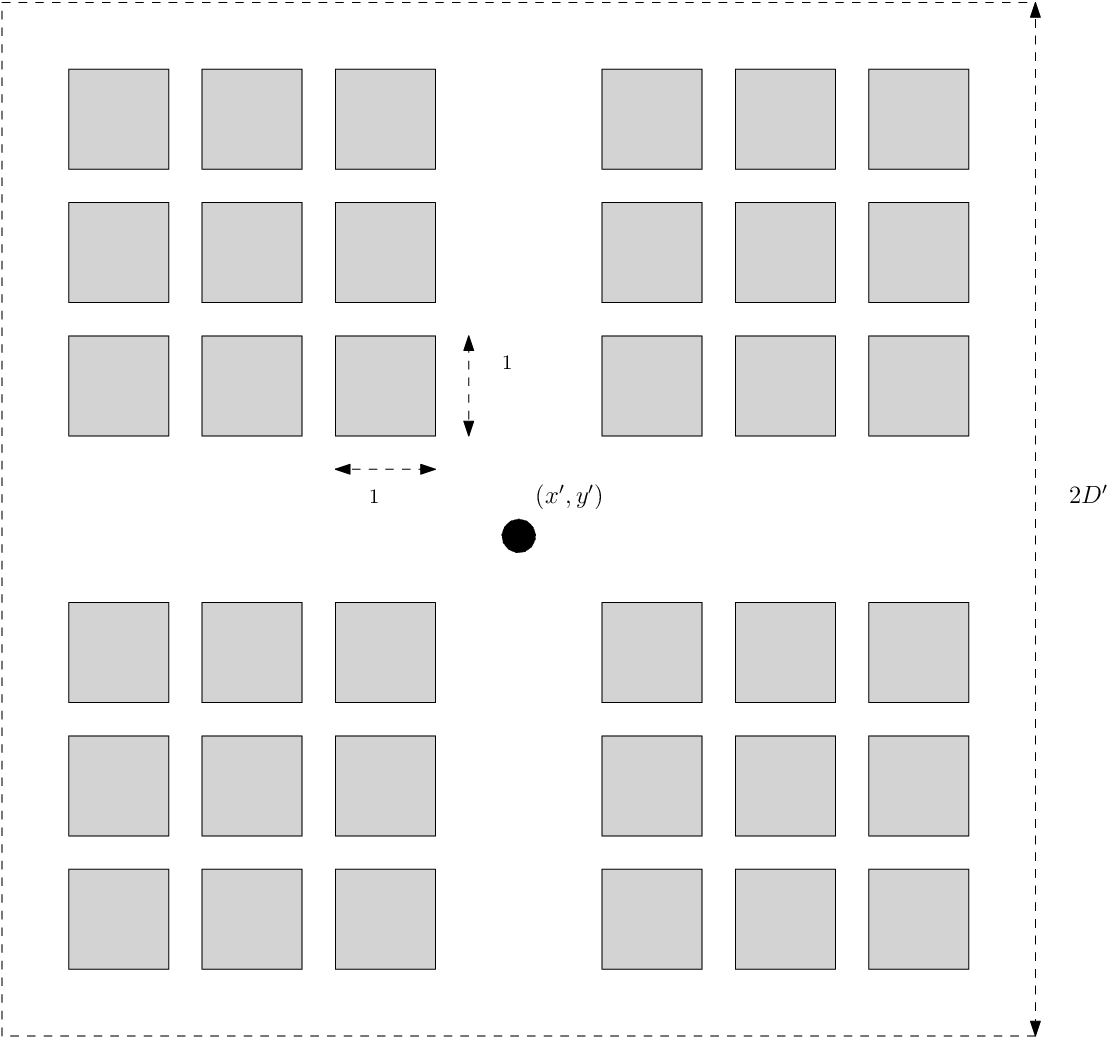}
    \caption{The disjoint union of unit squares in case $K=2$; the left-bottom corners of each square are the elements of $B$.}
    \label{fig1}
\end{figure}

Let $I=(i_1,i_2,\dots,i_{K-\ell})$ where $\ell$ is the smallest positive integer such that $2^{\ell-1}\ge s$, and $S''=\sum_{j\in I} Z_j$. Every step $Z_{j}$ where $j\in I$, can be in one of the four directions, creating~$4^{K-\ell}$ possible values. Since $a_{i_1}>1$ and the sequence is ``super-doubling'', no two such values lie in the same square with side $1$;   let us denote by $B$ the set of all possible values of~$S''$. As a result, for any $(x',y')\in\R^2$ 
$$
\bigcup_{(x,y)\in B} (x'-x,x'-x+1)\times(y'-y,y'-y+1)\subseteq 
(x'-\tilde D,x'+\tilde D)\times (y'-\tilde D,y'+\tilde D)
$$
is a collection of {\bf disjoint} unit squares (see Figure~\ref{fig1}); here
$$
\tilde D=1+a_{i_1}+a_{i_2}+\dots+a_{i_{K-\ell}}\le 2 a_{i_{K-\ell}}\le
2^{1-\ell}a_{i_K}=2^{1-\ell}a_{n_0}\le \frac{a_{n_0}}{s} \le D
$$  
where the last inequality follows from the $(r,s)$-monotonicity, since $n_1\ge r n_0$ implies that $a_i\ge a_{n_0}/s$ for each $i\ge  n_1$ and hence the same is true for $D$ by~\eqref{eqDn}. Moreover, note that~$S''$ has a uniform distribution over the set $B$ and
\begin{align*}   
\mathrm{card}(B)&=
4^{K-\ell}=4^{(1+\g_{n_0}) \log_2 a_{n_0}-\ell}=\frac{a_{n_0}^{2(1+\g_{n_0})}}{4^\ell}
\\
&\ge \frac{\left(\log \lfloor n/(2r)\rfloor\right)^{\left(\frac12+\epsilon\right)\cdot 2(1+\g_{n_0}) }}{4^\ell} >(\log n)^{1+\epsilon}
\end{align*}
where  the last inequality holds for all large $n$ since $|\g_{n_0}|\to0$ as the sequence $\{a_n\}$ is good.

Next, for every $(x',y')\in\R^2$
\begin{align*}
&\P\left(S'+S''\in (x',x'+1)\times (y',y'+1)\right)
\\=&\sum_{(x,y)\in B}\P(\{S''=(x,y)\}\cap\{S'\in (x'-x,x'-x+1)\times(y'-y,y'-y+1)\})
\\=&\sum_{(x,y)\in B}\P(S''=(x,y))\cdot \P(S'\in (x'-x,x'-x+1)\times(y'-y,y'-y+1))
\\ 
\le&\frac{1}{(\log n)^{1+\epsilon}}\cdot\sum_{(x,y)\in B}\P(S'\in (x'-x,x'-x+1)\times(y'-y,y'-y+1))
\\ 
=&\frac{1}{(\log n)^{1+\epsilon}}\cdot\P\left(S'\in \bigcup_{(x,y)\in B} (x'-x,x'-x+1)\times (y'-y,y'-y+1)\right)
\\ \le&\frac{1}{(\log n)^{1+\epsilon}}
         \cdot \P\left(S'\in (x'-\tilde D,x'+\tilde D)\times(y'-\tilde D,y'+\tilde D)\right)
         \leq\frac{c_1^2+o(1)}{n (\log n)^{1+\epsilon}}
    \end{align*}
by~\eqref{eqS'} and the fact that $\tilde D\le D$.

Finally, denoting by $S'''=\sum_{i\in\{1,\dots,\lfloor n/2\rfloor\}\setminus I}Z_i $ the  sum of the remaining steps, since $S',S'',S'''$ are independent, for all $(x'',y'')\in\R^2$ we get
\begin{align*}
\P\left(\sum_{i=1}^n Z_i\in(x'',x''+1)\times (y'',y''+1)\right)
&=
\P\left(S'+S''+S'''\in (x',x'+1)\times (y',y'+1)\right)
\\ &
\le \sup_{(x',y')\in\R^2} \P(S'+S''\in (x',x'+1)\times (y',y'+1))
\\ &\le \frac{ c_1^2+o(1)}{n(\log n)^{1+\epsilon}},
\end{align*}
which is summable over $n$ as $\epsilon>0$. The result now follows from the Borel-Cantelli lemma.
\end{proof}

The next statement tells us when we can apply Theorem~\ref{Tdouble}.
\begin{proposition}\label{propgood}
Suppose there is some $C>0$ such that $a_1=1$, $a_n\ge 1$, $|a_{n+1}-a_n|\le C$ for all $n$,  and $a_n\to\infty$. Then the sequence $a_n$ is good.
\end{proposition}
\begin{remark}
The above proposition covers the case when $a_n=n^\alpha$, $\alpha\in(0,1]$.
\end{remark}

\begin{proof}[Proof of Proposition~\ref{propgood}]
Fix an $n$ and let $j_1=n$. Now let us find $j_2$ such that $a_{j_2}\approx \frac12 a_{j_1}$. Since the distances between consecutive $a_i$s do not exceed $C$, we can find some $j_2$ such that
$$
\frac{a_{n}}2-C=\frac{a_{j_1}}2-C< a_{j_2}\le \frac{a_{j_1}}2.
$$
Similarly, we can find $j_3$ such that $a_{j_3}\le a_{j_2}/2$ and at the same time
$
 a_{j_3}> \frac{a_{j_2}}2-C\ge \frac{a_{n}}{2^2}-\frac C2-C.
$
By induction, we get that
$$
a_{j_k}\ge \frac{a_n}{2^{k-1}}-C\left(1+\frac12+\dots+\frac1{2^{k-2}}\right)=
\frac{a_n+2C}{2^{k-1}}-2C,
$$
and we can continue the process as long as $a_{j_k}\ge 1$. Hence, the maximum possible~$k$ for which the process is feasible is 
$$
\left\lfloor \log_2\left(  \frac{a_n+2C}{C+1/2}\right) \right\rfloor
=(1+\g_n)\log_2 a_n
$$
where $\g_n\to 0$ since $\log_2 a_n\to\infty$ as $n\to\infty$.
\end{proof}

\section{Recurrence for some general sequences}\label{SecR}

Let $B=\{b_1,b_2,\dots\}\subseteq \Z_+\setminus\{0\}$. We say that the set $B$ is {\em good} if for any $b\in B$ there are infinitely many coprimes of $b$ in $B$. 

\begin{theorem}\label{trec}
Suppose that $B\subseteq \Z_+\setminus\{0\}$ is good. Then there exists a sequence $a_1,a_2,\dots$   which contains all elements of the set $B$, and each element only finitely many times, such that the $a$-walk $S_n$ is recurrent.
\end{theorem}

\begin{lemma}\label{lemmaCF}
Let $C>0$ and $z\in\Z^2$ with $\norm{z}\le C$, and $b'$ and $b''$ are positive integers which are coprime. Then there exist three positive integers $c',c'', N_0$  such that $N_0$ depends on $c',c''$ and~$C$ only and
$$
\P(T_m=z\text{ for some }m\le N_0)\ge 1/2
$$
where
$$
T_m=\sum_{i=1}^m \tilde Z_i    \text{     with }
\tilde Z_i=b'(\xi_{i,1}'+\dots+\xi_{i,c'}')+b''(\xi_{i,1}''+\dots+\xi_{i,c''}'')
$$
and $\xi'_{i,k},\xi''_{i,k},$  are i.i.d.\ two-dimensional Rademacher random variables.  
\end{lemma}
\begin{proof}
Let $c',c''>0$ be such that $c'b'-c''b''=1$; the existence of such $c',c''$ follows from Bézout's identity. This implies that $\P(\tilde Z_i=\mathbf{f})> 0$  for each $\mathbf{f}=\be_1,-\be_1,\be_2,-\be_2$ and thus the Markov chain $T_m$ is irreducible on $\Z^2$.  

Next, by the Chung–Fuchs theorem (see \cite{Chung}) $T_m$ is recurrent and thus there exists $N_0$ such that 
$$
\inf_{z'\in \Z^2:\ \norm{z'}\le C} \P(T_m\text{ visits }z'\text{ for some }m\le N_0)\ge \frac12
$$
which concludes the proof.
\end{proof}

\begin{proof}[Proof of Theorem~\ref{trec}]
The construction of the sequence $a_i$ will be done recursively. Let $n_0=0$ and suppose that we have already constructed the sequence up to time $i=n_k$. Let $\alpha_k=\max\{a_1,a_2,\dots,a_{n_k}\}$, then $\norm{S_{n_k}}\le \alpha_k n_k$. Find the element of $B$ not already used, with the smallest index, call it $b'$, and let $b''$ be another element of $B$, not already used so far, such that $b'$ and $b''$ are coprimes. By Lemma~\ref{lemmaCF} there exist $c',c'',N_0$ corresponding to $b',b''$ and $C:=\alpha_k n_k$. Now define
$$
a_{n_k+i+1}=\begin{cases}
b',&\text{if } i=\ell\pmod{c'+c''}\text{ for some }\ell\in\{0,1,\dots,c'-1\} ,\\
b'',&\text{if } i=\ell \pmod{c'+c''} \text{ for some }\ell\in\{c',c'+1,\dots,c'+c''-1\} 
\end{cases}
$$
for $i=0,1,2,\dots, (c'+c'')N_0-1$. Define also $n_{k+1}=n_k+(c'+c'')N_0$. It is clear that by following this procedure, we will eventually use all elements of $B$, and, moreover, we will use each element of $B$ only finitely many times.

Now, by Lemma~\ref{lemmaCF}, since $\norm{S_{n_k}}\le C$,
$$
\P(S_n=0\text{ for some }n\in[n_k+1,n_{k+1}]\mid \F_{n_k})\ge \frac 12.
$$
Consequently, $S_n=0$ for infinitely many $n$s a.s., by Lévy's extension of the Borel-Cantelli lemma 
(see e.g.~\cite{WIL}).
\end{proof}

\subsection{An explicit recurrent sequence}
In this section, we present an explicit example of a sequence $a_n$ such that $a_n\to\infty$ and yet the walk is recurrent. The elements of this sequence will consist of blocks numbered $1,2,\dots$, and the $k$th block will be a consecutive sequence of exactly $k$ sub-blocks, denoted by $B_{k,1},B_{k,2},\dots,B_{k,k}$, where
\begin{align*}
B_{k,i}&=\underbrace{k,k,\dots,k,}_{L_{k,i}\text{ elements}}\, 
\underbrace{k-1,k-1,\dots,k-1}_{3k^2\text{ elements}}
\quad\text{for }i=1,2,\dots,k-1;
\\
B_{k,k}&=\underbrace{k,k,\dots,k}_{L_{k,k}\text{ elements}},
\qquad \text{where }L_{k,i}=2^{2^{k^2+i-1}},\ i=1,2,\dots,k.
\end{align*}

Now suppose that the sequence ${a_n}$ is given by $\{a_n\}=B_1 B_2\dots$ where $B_k=B_{k,1} B_{k,2}\dots B_{k,k}$.
\begin{theorem}\label{Tex}
A two-dimensional Rademacher walk with the above sequence is recurrent.
\end{theorem}
We begin with a (probably) well-known result.
\begin{lemma}\label{lemmacube}
Consider a simple random walk $T_n=(X_n,Y_n)$ on $\Z^2$ with steps $\pm\be_1,\pm\be_2$ and $T_0=(x_0,y_0)$. Let $\tau_0=\inf\{n:\ T_n=\mathbf{0}\}$. Then there exists a universal constant $\alpha>0$ such that
$$
\P(\tau_0\le r^3)\ge \alpha, \qquad\text{where }r=\norm{(x_0,y_0)}.
$$
\end{lemma}
\begin{proof}
The fact that $\tau_0<\infty$ a.s.\ follows from the recurrence of the two-dimensional walk. Let $R=r^{1.25}$ and
$$
\tau=\inf\{n:\ T_n=\mathbf{0} \text{ or }\norm{T_n}\ge R\}.
$$
Define
$$
f(x,y)=\begin{cases}\log(x^2+y^2-1/2), \text{if } (x,y)\ne\mathbf{0};\\
-5,&\text{otherwise}.
\end{cases}
$$
An easy calculation\footnote{
Indeed, let $\Delta_{x,y}:=\frac14(f(x+1,y)+f(x-1,y)+f(x,y+1)+f(x,y-1))-f(x,y)$.
Then $\exp(4\Delta_{x,y})=1-\frac{64(x^2-y^2)^2}{(2x^2+2y^2-1)^4}\le 1$ when  $|x|+|y|\ge 2$, and 
$\exp(4\Delta_{x,y})=126 e^{-5}<0.85$  when $|x|+|y|=1$.
}
shows that $\xi_n=f(T_n)$  is a supermartingale as long as $ T_n\ne\mathbf{0}$, hence by the optional stopping theorem, as $r\to\infty$,
\begin{align}\label{tau0}
\P(T_\tau\ne \mathbf{0})\le \frac{\ln(r)}{\ln(R)}+o(1)=\frac45+o(1).    
\end{align}

On the other hand, 
$$
\P(\tau> r^3)=o(1)\qquad\text{as }r\to\infty.
$$
Indeed, since during $3 \lfloor  R^{2.2}\rfloor\ll r^3 $ steps, with probability very close to $1$, the walk $T_n$ will make at least $\lfloor R^{2.2}\rfloor $ horizontal steps. Then, by the reflection principle,
$$
\P(\max_{i\le n} (X_i-X_0)\ge m)=2\, \P(X_n-X_0>m)+\P(X_n-X_0=m).
$$
 By choosing $m=\lceil R\rceil$ and $n=\lfloor R^{2.2}\rfloor$ we get
$$
\P(\tau> r^3)\le \P\left(\max_{i\le r^3} \norm{T_i} <R\right)\le \P\left(\max_{i\le n} (X_i-X_0)<m\right)=\P\left(|\eta|<R^{-0.1}\right)+o(1)=o(1)
$$
where $\eta\sim\mathcal{N}(0,1)$.

Finally,
\begin{align*}
\P(\tau_0>r^3)&=\P(\tau_0>r^3,\tau< \tau_0)+\P(\tau_0>r^3,\tau\ge  \tau_0)
\\
&\le \P(\tau< \tau_0)+\P(\tau>r^3)
=\P(T_\tau\ne \mathbf{0})+o(1)\le\frac45+o(1)
\end{align*}
by~\eqref{tau0}.
\end{proof}

\begin{proof}[Proof of Theorem~\ref{Tex}]
Let $n_k$ be the index of the last element in the $k$th block and $n_{k,j}$  be the index of the last element in the $j$th sub-block of the $k$th block. 

Let $E_{k,j}=\left\{\norm{S_{n_{k,j}}}< k \sqrt{L_{k,j}\log L_{k,j} }\right\}$. Since $L_{k,j}=L_{k,j-1}^2$ for $j\ge 2$, and also $L_{k+1,1}\gg L_{k,k}^2$, by Hoeffding's inequality
\begin{align}\label{eqEj}
\P(E_{k,j}^c)\le \frac{2+o(1)}{L_{k,j}^2},
\qquad j=2,3,\dots,k.
\end{align}
Since 
$
L_{k,j+1}=L_{k,j}^2>\left(\sqrt{L_{k,j}\log L_{k,j}}\right)^3,
$
on $E_{k,j}$ by Lemma~\ref{lemmacube}, with probability at least $\alpha$, the walk $T_n$, $n\in [n_{k,j},n_{k,j+1}-3k^2]$ will visit the square $[0,k-1]^2$ at least once (as this is essentially a two-dimensional simple random walk with step sizes equal to~$k$.)

Now, Corollary~3.2 from~\cite{BV} implies that with probability of order $\frac{1}{k^2}$, the walk starts this sub-block at a position whose $x-$ and $y-$coordinates are both divisible by $k$. This follows from the fact that, with high probability, at least $k^2$ steps were taken in each of the two directions during the last 
$3k^2$ steps of the preceding block. As a result, 
$$
\P(A_{k,j+1}\mid {E}_{k,j})\ge \frac{c+o(1)}{k^2}
\qquad \text{where } A_{k,j+1}=\{T_n=\mathbf{0}\text{ for some }
n\in[n_{k,j},n_{k,j+1})\}
$$
for some $c>0$ independent of $k$ and $j$.

We will use the extension of the conditional Borel-Cantelli lemma from~\cite{BV} to conclude recurrence. Define $\mathcal{F}_{k,j}$ as the sigma-algebra generated by the walk up to the time $n_{k,j}$. We have 
$$
\sum_k\sum_{j=1}^{k-1}\P(A_{k,j+1}\mid E_{k,j},\mathcal{F}_{k,j})\ge  \sum_k \sum_{j=1}^{k-1}\frac{c+o(1)}{k^2}=\infty.
$$
At the same time,
$$
\sum_k \sum_{j=1}^{k-1}\P(E_{k,j}^c)<\infty.
$$
by~\eqref{eqEj}.
Hence, the events $A_{k,j+1}$ occur infinitely often a.s.\ by Lemma~5.1 in~\cite{BV}, thus implying recurrence of site $\mathbf{0}$. The fact that the walk visits all other $\mathbf{w}\in\Z^2$ a.s.\  follows from the same argument.
\end{proof}

\begin{remark}
The result of Theorem~\ref{trec} can be extended for faster growing sequences, for example, with the sub-block sizes equal to $2^{2^{k^{1+\nu}+j}}$, for  $\nu\in(0,1)$.
\end{remark}

\subsection*{Acknowledgment}
The research is partially supported by the Swedish Science Foundation grant VR 2019-04173. S.B.~would like to acknowledge Tom Johnston for suggesting the idea of the super-doubling sequence. 



\end{document}